\documentclass[11pt,a4paper]{amsart}
\usepackage[latin1]{inputenc}
\usepackage[english]{babel}
\usepackage[T1]{fontenc}
\usepackage{amsmath,amssymb}
\usepackage{stmaryrd}
\usepackage{enumerate}
\usepackage{amsthm}
\usepackage{mathrsfs}
\usepackage{geometry}
\usepackage{mathtools}
\usepackage{graphicx}
\usepackage{array}
\usepackage[draft,english]{fixme}
\usepackage[all,line]{xy}
\usepackage[usenames,dvipsnames,svgnames,table]{xcolor}
\usepackage{xfrac}
\mathtoolsset{showonlyrefs}
\newtheorem{thm}{Theorem}[section]
\newtheorem{prop}[thm]{Proposition}
\newtheorem{lemma}[thm]{Lemma}
\theoremstyle{definition}
\newtheorem{defin}[thm]{Definition}
\theoremstyle{remark}
\newtheorem{rem}[thm]{Remark}

\newtheorem{cor}[thm]{Corollary}

\newcommand{\iso}{\simeq}
\newcommand{\isomap}{\stackrel{\sim}{\to}}

\newcommand{\C}{\mathbb{C}}

\newcommand{\Z}{\mathbb{Z}}

\newcommand{\g}{\mathfrak{g}}

\newcommand{\bfr}{\mathfrak{b}}

\newcommand{\niltil}{\tilde{\mathcal{N}}}

\newcommand{\Sym}{\operatorname{Sym}}
\newcommand{\Acal}{\mathcal{A}}
\newcommand{\Bcal}{\mathcal{B}}
\newcommand{\Ccal}{\mathcal{C}}
\newcommand{\Ocal}{\mathcal{O}}
\newcommand{\Pcal}{\mathcal{P}}
\newcommand{\Ecal}{\mathcal{E}}
\newcommand{\Fcal}{\mathcal{F}}
\newcommand{\Dcal}{\mathcal{D}}
\newcommand{\Mcal}{\mathcal{M}}
\newcommand{\Ncal}{\mathcal{N}}
\newcommand{\Rcal}{\mathcal{R}}
\newcommand{\Tcal}{\mathcal{T}}
\newcommand{\Scal}{\mathcal{S}}
\newcommand{\Hcal}{\mathcal{H}}
\newcommand{\Xcal}{\mathcal{X}}
\newcommand{\Ycal}{\mathcal{Y}}
\newcommand{\Vcal}{\mathcal{V}}
\newcommand{\Wcal}{\mathcal{W}}
\newcommand{\Ical}{\mathcal{I}}
\newcommand{\Jcal}{\mathcal{J}}
\newcommand{\Gcal}{\mathcal{G}}

\newcommand{\QCoh}{\operatorname{QCoh}}
\newcommand{\Coh}{\operatorname{Coh}}
\newcommand{\Hom}{\operatorname{Hom}}

\newcommand{\pr}{{\operatorname{pr}}}

\newcommand{\coker}{\operatorname{coker}}

\newcommand{\module}{\operatorname{mod}}
\newcommand{\Res}{\operatorname{Res}}
\newcommand{\Ind}{\operatorname{Ind}}

\newcommand{\id}{\operatorname{Id}}

\newcommand{\MF}{\operatorname{MF}}
\newcommand{\Perf}{\operatorname{Perf}}
\newcommand{\symcal}{\mathcal{S}\! \operatorname{ym}}

\numberwithin{equation}{section}
 
\begin{document}

\title{Equivariant Matrix Factorizations and Hamiltonian reduction}
\author{Sergey Arkhipov}
\author{Tina Kanstrup}
\address{S.A. Matematisk Institut, Aarhus Universitet, Ny Munkegade, DK-8000 , Aarhus C, Denmark, email: hippie@math.au.dk} 
\address{T.K. Centre for Quantum Geometry of Moduli Spaces, Aarhus Universitet, Ny Munkegade, DK-8000 , Aarhus C, Denmark, email: tina.kanstrup@mail.dk}

\maketitle

\begin{abstract}
Let $X$ be a smooth scheme with an action of an algebraic  group $G$. We establish an equivalence of two categories related to the corresponding moment map $\mu : T^*X \to \g^*$ - the derived category of $G$-equivariant coherent sheaves on the derived fiber $\mu^{-1}(0)$ and the derived category of $G$-equivariant matrix factorizations on $T^*X \times \g$ with potential given by $\mu$.
\end{abstract}

\section{Motivation}
\subsection{ Affine Hecke category and derived loop groups.}
We start from an informal discussion of a material which is supposed to explain our interest in the category of equivariant matrix factorizations. We outline the context as well as the place
of our paper in a bigger project. 

The present paper is the third one in the series (see \cite{AK1}, \cite{AK2}). Recall that for a reductive algebraic group $G$ with a Borel subgroup $B$
we considered a monoidal triangulated category $\text{QCHecke}(G,B)$ -- the derived category of quasi-coherent sheaves on the stack $B\backslash G/B$. The category is a higher analog
of the affine 0-Hecke algebra considered by Kostant and Kumar in [KK]. 

Given a smooth scheme $X$ with an action of $B$ we constructed a monoidal action of $\text{QCHecke}(G,B)$ on the derived
category of $B$-equivariant quasi-coherent sheaves on $X$. The action is a categorification of the Demazure operators acting on the equivariant Grotehndieck group $K_B(X)$ and
consturcted by Harada et al.

A natural challenge is to upgrade the action of the categorified  affine 0-Hecke algebra to an action of a  categorification of the usual affine Hecke algebra. Such categorification is known under the 
name of the affine Hecke category  $\text{Hecke}_{\text{aff}}$. One realization of the affine Hecke category is due to Bezrukavnikov who identified $\text{Hecke}_{\text{aff}}$ with the derived category of
$G$-equivariant coherent sheaves on the Steinberg variety $St_G$. 

Ben-Zvi and Nadler noticed that the latter category has a meaning in the realm of Derived Algebraic Geometry puting it in the context strikingly similar to the one for $\text{QCHecke}(G,B).$
Namely, one considers the derived group schemes of topological  loops with values in $G$ (resp., in $B$) denoted by $L_{\text{top}}(G)$ (resp., by $L_{\text{top}}(B))$. Ben-Zvi and Nadler
interpret $\text{Hecke}_{\text{aff}}$ as $\text{QCHecke}(L_{\text{top}}(G),L_{\text{top}}(B))$. Thus to imitate formally our consturction of Demazure functors for a $B$-scheme $X$, one should define
a category of $L_{\text{top}}(B)$-equivariant coherent sheaves on $L_{\text{top}}(X)$.

The goal of the present paper is to work out a down-to-Earth approach to a category playing the role of 
$D^b(\Coh^{L_{\text{top}}(B)}(L_{\text{top}}(X)))$. 
In the fourthcoming paper 
\cite{AK4} we will construct a categorical braid group action on our model for 
$D^b(\Coh^{L_{\text{top}}(B)}(L_{\text{top}}(X)))$. 

\subsection{From derived loop groups to  Hamiltonian reduction.}
Let us informally outline the ideas leading to our definition of the category in question.
\begin{enumerate}
\item
 For a smooth scheme $X$, the structure sheaf for $L_{\text{top}}(X)$ is known to be the Hochschild homology complex for $\mathcal{O}_X$. Hochschild-Kostant-Rosenberg theorem
identifies the complex with the sheaf of DG-algebras of differential forms $\Omega_X$, with zero differential. Thus one can think of the category $\Coh(L_{\text{top}}(X))$ as of the category of
DG-modules over $\Omega_X$. Koszul duality provides a bridge from this category to the derived category of coherent sheaves on $T^*X$.
\item
 To impose a $L_{\text{top}}(B)$-equivariance structure on an object of $D^b(\Coh(T^*X))$ we notice that the DG-coalgebra $\mathcal{O}_{L_{\text{top}}(B)}$ is a semi-direct product of 
$\mathcal{O}_B$ and of $\Lambda(\mathfrak{b}^*)$. The $\mathcal{O}_B$-coaction stands for a $B$-equivariance structure. The coaction of the exterior coalgebra after dualisation becomes 
a structure of a $\Lambda(\mathfrak{b})$-DG-module. 
\item
 Summing up, we obtain a $B$-equivariant sheaf of DG-algebras 
$\Lambda(\mathfrak{b})\otimes \mathcal{O}_{T^*X}$ 
and the category of $B$-equivariant sheaves of DG-modules
over it. Notice that the coaction of the exterior coalgebra on $\mathcal{O}_{T^*X}$ is non-trivial. This results in a differential in
$\Lambda(\mathfrak{b})\otimes \mathcal{O}_{T^*X}$ which can be identified with the Koszul type differential on
$$
\left(\Lambda(\mathfrak{b})\otimes 
\Sym(\mathfrak{b})
\right)\otimes_{\Sym(\mathfrak{b})}
\mathcal{O}_{T^*X}.
$$ 
Consider the moment map $\mu:\ T^*X\to
 \mathfrak{b}^*$.
The  sheaf of DG-algebras  above is a representative for the structure sheaf of $\mu^{-1}(0)$ understood as the derived fiber product $T^*X\times_{ \mathfrak{b}^*}\{0\}$.
We end up with a category which has a completely recognizable flavour. Namely, we define the category $D^b(\Coh^{L_{\text{top}}(B)}(L_{\text{top}}(X)))$ to be the derived category 
of coherent sheaves on the stack
$\mu^{-1}(0)/B$.

\end{enumerate}

Notice that $\mu^{-1}(0)$ is a derived scheme, but now it is described very explicitly as an equivariant quasi-coherent sheaf of DG-algebras on $T^*X$. The definition of the derived category of 
equivariant DG-modules over it does not require any use of infinity-categories, so the story belongs to Derived Algebraic Geometry in the mildest possible sense.

\subsection{Hamiltonian reduction and equivariant matrix factorizations.} Recall that in Poisson Geometry, Hamiltonian reduction is a way to describe the cotangent bundle to the quotient
variety $X/G$ for a $G$-variety $X$ with an action of a Lie group $G$ which is good enough, e.g. close to be free. One considers the moment map $\mu:\ T^*X\to \mathfrak{g}^*$, hopes that
the $G$-action on $\mu^{-1}(0)$ is as good as possible (this is certainly true e.g. in the case when $X\to X/G$ is a principal $G$-bundle) and identifies $T^*(X/G)$ with the Hamiltonian
reduction space $\mu^{-1}(0)/G$.

The moment map plays the crucial role in our considerations too. Let $G$ be an algebraic group acting on a smooth scheme $X$. We have the following three (equivalent)  incarnations of the moment map
for the $G$-action:
\begin{enumerate}
\item
The moment map $\mu$ comes from the map of rings 
$\Sym(\mathfrak{g})\to \Gamma(T^*X,\mathcal{O}_{T^*X})$.
\item
The moment map is a morphism from $\mathfrak{g}$ to the vector space of global vector fields on $X$. This leads to a differential in the sheaf of graded algebras $\Lambda(\mathfrak{b})\otimes \mathcal{O}_{T^*X}$. Notice that this way we obtain the concrete model for the structure sheaf of the derived scheme $\mu^{-1}(0)$ discussed above.
\item
The moment map times identity is a $G$-equivariant morphism $T^*X \times \g \to \g^* \times \g$. Composing with the natural pairing this defines a $G$-invariant global function on $T^*X \times \mathfrak{g}$.
\end{enumerate}

Each of the incarnations of the moment map above leads to a triangulated category. 
\begin{enumerate}
\item
Suppose that $X$ is a principal $G$-bundle over $Y=X/G$. One considers the derived category of coherent sheaves on the Hamiltonian reduction of $T^*X$ which is isomorphic to $T^*Y$.
\item
More generally, using the second incarnation of the moment map, one defines the derived category of $G$-equivariant DG-modules over the sheaf of  DG-algebras 
$\Lambda(\mathfrak{b})\otimes \mathcal{O}_{T^*X}$.

\item
Given a $G$-scheme $Z$ and a global $G$-invariant function $h$ on it one defines the derived category of equivariant matrix factorizations with the potential $h$.
\end{enumerate}
The present paper establishes the equivalences between the three approaches and the three triangulated categories.

\subsection{The structure of the paper.} In Section 2, for a scheme $X$ with an action of an algebraic group $G$,  we recall the setting of $G$-equivariant sheaves of DG-algebras on $X$.
We define the corresponding derived categories of $G$-equivariant quasicoherent sheaves of DG-modules and the functors of inverse and direct image for them. 

In Section 3, we recall the equivariant Linear Koszul Duality due to Mirkovic and Riche. This is the main technical tool to connect the categories 2) and 3) discussed above.

In Section 4, we consider a more general setup. Let $X$ be a smooth complex variety with an action of a reductive algebraic group and  let $\pi : E \to X$ be a $G$-equivariant vector bundle. Fix a regular $G$-equivariant global section $s$ of that vector bundle. The zero scheme of $s$ is denoted by $Y$. The dual vector bundle is denoted by $\pi^\vee : E^\vee \to X$. This defines a $G$-invariant function
\begin{gather} 
W: E^\vee \to E^\vee \times_X E \stackrel{\langle \;, \: \rangle}{\longrightarrow} \C,\\
 a_x \mapsto \langle ((\pi^\vee)^*s)(a_x), a_x \rangle.
\end{gather}
Using equivariant Linear Koszul duality we establish an equivalence between the equivariant singularity category $D_{\text{sg}}^{G \times \C^*}(W^{-1}(0))$ and $D^b(\Coh^G(Y))$. This is an extension to the equivariant setting of a result by Isik.
 
In Section 5, we recall the definition of the absolute derived category of matrix factorizations on a stack $X$ with potential $W$. It is denoted by $\text{DMF}(X, W)$. Polishchuk and Vaintrob proved that under some assumptions $\text{DMF}(X, W)$ is equivalent to $D_{\text{sg}}(W^{-1}(0))$.

In Section 6, we combine the results in section 4 and 5 and obtain an equivalence of categories
\[ \operatorname{DMF}_{G \times \C^*}(E^\vee,W) \iso D^b(\Coh^G(Y)). \]
This also hold for non-reductive $G$. Lastly, we apply this result  to the Hamiltonian reduction setting where $E^\vee$ is the trivial vector bundle $\pi : T ^*X \times \g^* \to T^*X$ and the section is the moment map. This establishes the equivalence between the categories 1) and 3)
\[ \operatorname{DMF}_{G \times \C^*}(T^*X \times \g,W) \iso D^b(\Coh(T^*Y)). \]

\subsection{Acknowledgments}
We would like to thank V. Baranovsky, A. I. Efimov and A. Polishchuk for inspiring discussions.

Both authors' research was supported in part by center of excellence grants "Centre for Quantum Geometry of Moduli Spaces" and by FNU grant "Algebraic Groups and Applications".

\section{Sheaves of DG-algebras and DG-modules over them}

We recall the basic definitions of equivariant quasi-coherent sheaves of differential graded modules.

\begin{defin}
Let $\Acal=\bigoplus_{p \in \Z} \Acal^p$ be a sheaf of $\Z$-graded $\Ocal_X$-algebras on a complex algebraic variety $X$. Denote the multiplication map by $\mu_\Acal : \Acal \otimes_{\Ocal_X} \Acal \to \Acal$.
\begin{enumerate}
\item The sheaf $\Acal$ is a sheaf of dg-algebras  if it is provided with an endomorphism of $\Ocal_X$-modules $d_\Acal :\Acal \to \Acal$ of degree 1, such that $d_\Acal \circ d_\Acal=0$, satisfying the following formula on $\Acal^i \otimes \Acal$ for any $i \in \Z$:
\[ d_{\Acal} \circ \mu_{\Acal}=\mu_{\Acal} \circ (d_\Acal \otimes \id_{\Acal})+(-1)^i \mu_{\Acal} \circ (\id_{\Acal^p} \otimes d_\Acal). \]

\item A morphism of sheaves of dg-algebras on the same scheme is a morphism of sheaves of graded algebras commuting with the differentials.

\item A morphism of dg-algebras on different schemes $f : (X,\Acal) \to (Y,\Bcal)$ is the data of a morphism of schemes $f_0 : X \to Y$, and a morphism of sheaves of dg-algebras $f_0^*\Bcal \to \Acal$.

\item Define the opposite dg-algebra $\Acal^{op}$ to have the same elements and differential as $\Acal$ but a new multiplication $a \circ b:=(-1)^{\deg(a) \deg(b)}ba$. The sheaf of dg-algebras $\Acal$  is called graded-commutative if the identity map $\id : \Acal \to \Acal^{op}$ is an isomorphism of sheaves of dg-algebras.

\item A $\Acal$-dg-module is a sheaf of $\Z$-graded left $\Acal$-modules $\Fcal$ on $X$ together with an endomorphism of $\Ocal_X$-modules $d_\Fcal : \Fcal \to \Fcal$ of degree 1, such that $d_\Fcal \circ d_\Fcal = 0$ and satisfying the following formula on $\Acal^i \otimes_{\Ocal_X} \Fcal$ for $i \in \Z$, where $\alpha_\Fcal :\Acal \otimes_{\Ocal_X} \Fcal \to \Fcal$ is the action map:
\[d_\Fcal \circ \alpha_\Fcal =\alpha_\Fcal \circ (d_\Acal \otimes \id_\Fcal)+(-1)^i \alpha_\Fcal \circ (\id_{\Acal^i} \otimes d_\Fcal).\]

\item A morphism of $\Acal$-dg-modules is a morphism of sheaves of graded $\Acal$-modules commuting with differentials.

\item A quasi-coherent dg-sheaf $\Fcal$ on $(X,\Acal)$ is an $\Acal$-dg-module such that $\Fcal^i$ is a quasi-coherent $\Ocal_X$-module for all $i \in \Z$.
\end{enumerate}
\end{defin}

\begin{defin}
Let $G$ be a complex reductive algebraic group acting on a complex algebraic variety $X$. Let $\Acal$ be a sheaf of dg-algebras on $X$ and assume that $\Acal_i$ is $G$-equivariant for all $i \in \Z$ and that the multiplication and differential are $G$-equivariant. A $\Acal$-dg-module $\Fcal$ is $G$-equivariant if $\Fcal_i$ is $G$-equivariant for all $i \in \Z$ and the differential and action morphisms are $G$-equivariant.
\end{defin}

The category of $G$-equivariant quasi-coherent left dg-modules over the dg-algebra $\Acal$ is denoted by $\Ccal\QCoh^G(\Acal)$. The definition of the equivariant derived category is analogous to the non-equivariant derived category as defined in \cite{BL}. We recall the definitions.

\begin{defin}
\begin{enumerate}
\item Two morphisms $f,g : \Mcal \to \Ncal$ in $\Ccal\QCoh^G(\Acal)$ are homotopic if there exists a morphism of modules over the graded ring $\Acal$ (but not necessarily a morphism of $\Acal$-modules) $s : \Mcal \to \Ncal[-1]$ s.t.
\[ f-g=s d_\Mcal+d_\Ncal s. \]
We write $f \sim g$.
\item The homotopy category $H^0(\QCoh^G(\Acal))$ has the same objects as $\Ccal\QCoh^G(\Acal)$ and morphisms
\[ \Hom_{H^0(\QCoh^G(\Acal))}(\Mcal,\Ncal) :=\Hom_{\Ccal\QCoh^G(\Acal)}(\Mcal,\Ncal)/\{\text{morphisms }\sim 0\}. \]
\item The cohomology of $\Mcal \in \Ccal\QCoh^G(\Acal)$ is the graded sheaf of $\Ocal_X$-modules $\Hcal(\Mcal)=\ker(d_\Mcal)/\text{im}(d_\Mcal)$. $\Mcal$ is acyclic if $\Hcal(\Mcal)=0$.
\item A morphism is a quasi-isomorphism if it induces an isomorphism on cohomology. The derived category $\Dcal\QCoh^G(\Acal)$ is the localization of $H^0(\QCoh^G(\Acal))$ with respect to quasi-isomorphisms.
\item A coherent dg-module $\Mcal$ over $\Acal$ is a quasi-coherent dg-sheaf whose cohomology sheaf $\Hcal(\Mcal)$ is coherent over $\Hcal(\Acal)$. The full subcategory of $\Dcal\QCoh^G(\Acal)$ whose objects are coherent is denoted by $\Dcal\Coh^G(\Acal)$.
\item The full subcategory of $\Dcal\QCoh^G(\Acal)$ consisting of objects whose cohomology is bounded and coherent as a $\Ocal_X$-module is denoted by $\Dcal^{bc} \QCoh^G(\Acal)$.
\end{enumerate}
\end{defin}

\begin{rem}  \label{OX}
Consider $\Ocal_X$ as a dg-algebra with $\Ocal_X$ in degree zero and 0 elsewhere. Then
\[ \Dcal\QCoh^G(\Ocal_X) \iso D^b(\QCoh^G(X)), \qquad \Dcal\Coh^G(\Ocal_X) \iso D^b(\Coh^G(X)). \]
\end{rem}

\subsection{Functors}
Let $G$ be a reductive algebraic group acting on a complex algebraic variety $X$. To be able to define the derived functors we will assume that the following property hold:
\begin{gather}
\text{For any } \Fcal\in \Coh^G(X), \text{ there exists } \Pcal \in \Coh^G(X)\\
\text{which is flat over } \Ocal_X \text{ and a surjection } \Pcal \twoheadrightarrow \Fcal \text{ in } \Coh^G(X). \label{Assump}
\end{gather}

\begin{rem} \label{AmpleProp}
Property \ref{Assump} is satisfied e.g. when $X$ admits an ample family of line bundles in the sense of \cite[Definition 1.5.3]{VV} or when $X$ is normal and quasi-projective (see \cite[Proposition 5.1.26]{CG}).
\end{rem}

\begin{defin}
Let $\Acal$ be an equivariant sheaf of dg-algebras on $X$. If $\Acal$ is quasi-coherent, non-positively graded and graded-commutative then the pair $(X,\Acal)$ is called a dg-scheme.
\end{defin}

From now on we will always assume that we are working with a dg-scheme. In particular, the category of left $\Acal$-dg-modules is equivalent to the category of right $\Acal$-dg-modules (see \cite[10.6.3]{BL}). Furthermore, we will assume that $\Acal$ is locally finitely generated over $\Acal^0$, that $\Acal^0$ is locally finitely generated as an $\Ocal_X$-algebra, and that $\Acal$ is K-flat as a $\mathbb{G}_m$-equivariant $\Acal^0$-dg-module. 

The last assumption is justified by the following observation in \cite[Section 2.2]{MR3}: If $A$ is the $G$-equivariant affine scheme over $X$ such that the push-forward of $\Ocal_A$ to $X$ is $\Acal^0$, then there exists a $\mathbb{G}_m \times G$-equivariant quasi-coherent $\Ocal_A$-dg-algebra $\Acal'$ whose direct image to $X$ is $\Acal$ and there is an equivalence of categories $\Ccal \QCoh^G(\Acal') \iso \Ccal \QCoh^G(\Acal)$. Using this trick one can always reduce to the situation in which $\Acal$ is $\Ocal_X$-coherent and $K$-flat as an $\Ocal_X$-dg-module.

\begin{lemma} \cite[Lemma 2.5]{MR3}\label{EquiHom}
Let $\Acal$ be as above. For any $\Mcal,\Ncal \in \Dcal^{bc}(\QCoh^{G \times \C^*}(\Acal))$ the $\C$-vector space $\Hom_{\Dcal^{bc}(\QCoh^{\C^*}(\Acal))}(\Mcal,\Ncal))$ has a natural structure of an algebraic $G$-module. Moreover, the natural morphism
\[ \Hom_{\Dcal^{bc}(\QCoh^{G \times \C^*}(\Acal))}(\Mcal,\Ncal) \to (\Hom_{\Dcal^{bc}(\QCoh^{\C^*}(\Acal))}(\Mcal,\Ncal))^G \]
induced by the forgetful functor is an isomorphism.
\end{lemma}

One can define the usual functors on $\Ccal \QCoh^G(\Acal)$. We define the internal $\Hcal om$ functor
\[ \Hcal om_\Acal^G(-,-) : \Ccal \QCoh^G(\Acal) \times \Ccal \QCoh^G(\Acal) \to \Ccal \QCoh^G(\Ocal_X) \]
For $\Mcal, \Ncal \in \Ccal \QCoh^G(\Acal)$ the sheaf of $\Ocal_X$-dg-modules $\Hcal om^G_\Acal(\Mcal,\Ncal)$ is the graded sheaf of $\Ocal_X$-modules with the $i$-th component being local equivariant homomorphisms of graded $\Acal$-modules $\Mcal \to \Ncal[i]$ (not necessarily commuting with the differentials). For $\phi \in \Hcal om^G_\Acal(\Mcal,\Ncal)^i$ the differential is given by
\[ d(\phi)=d_\Mcal \circ \phi-(-1)^i \phi \circ d_\Ncal. \]

We also have a tensor product
\[ - \otimes_{\Acal} - : \Ccal \QCoh^G(\Acal) \times \Ccal \QCoh^G(\Acal) \to \Ccal \QCoh^G(\Ocal_X) \] 
The sheaf of $\Ocal_X$-dg-modules $\Fcal \otimes_{\Acal} \Gcal$ is graded in the natural way and on local sections of $\Fcal^i \otimes_\Acal \Gcal$ the differential is given by
\[ d(f \otimes g)=d(f)\otimes g+(-1)^i f \otimes d(g). \]
It is equivariant with respect to the diagonal $G$-action.\\

Let $f:(X,\Acal) \to (Y, \Bcal)$ be a $G$-equivariant morphism of dg-schemes. This defines a the morphism of sheaves of dg-algebras since, by adjunction, the morphism $f^* \Bcal \to \Acal$ corresponds to a morphism $\Bcal \to f_* \Acal$. We define the direct image functor to be restriction of scalars using this map.
\[ f_* : \Ccal \QCoh^G(\Acal) \to \Ccal \QCoh^G(\Bcal). \]
We can also define an inverse image functor using the tensor product
\begin{gather} 
f^* : \Ccal\QCoh^G(\Bcal) \to \Ccal\QCoh^G(\Acal), \\
 \Fcal \to \Acal \otimes_{f^* \Bcal} f^* \Fcal.
 \end{gather}

\begin{lemma} \cite[Lemma 2.7 and Prop. 2.8]{MR3}
Assume that $(X,G)$ satisfies the above assumptions and let $f: (X,\Acal) \to (Y,\Bcal)$ be a $G \times \mathbb{G}_m$-equivariant morphism of dg-schemes. Then
\begin{enumerate}
\item For any object $\Mcal \in \Ccal \QCoh^G(\Bcal)$, there exits an object $\Pcal \in \Ccal \QCoh^G(\Bcal)$, which is $K$-flat as a $\Bcal$-dg-module and a quasi-isomorphism of $G \times \mathbb{G}_m$-equivariant $\Bcal$-dg-modules $\Pcal \to \Mcal$.
\item The functor of pull-back admits a derived functor
\[ Lf^* : \Dcal \QCoh^G(\Bcal) \to \Dcal \QCoh^G(\Acal) \]
and the following diagram is commutative
\[ \xymatrix{\Dcal \QCoh^G(\Bcal) \ar[d]^{\text{For}} \ar[r]^{Lf^*} & \Dcal \QCoh^G(\Acal) \ar[d]^{\text{For}} \\ 
\Dcal \QCoh^G(\Bcal) \ar[r]^{Lf^{*}} & \Dcal \QCoh^G(\Acal)} \]
\item For any $\Ncal \in \Ccal^+\QCoh^G(\Acal)$, there exists an object $\Ical \in \Ccal^+\QCoh^G(\Acal)$ which is $K$-injective in $\Ccal \QCoh^G(\Acal)$ and a quasi-isomorphism $\Ncal \to \Ical$.
\item The functor of push-forward admits a derived functor
\[Rf_*: \Dcal^+\QCoh^G(\Acal) \to \Dcal \QCoh^G(\Bcal)\]
and the following diagram is commutative up to isomorphism
\[ \xymatrix{\Dcal^+\QCoh^G(\Acal) \ar[d]^{\text{For}} \ar[r]^{Rf_*} & \Dcal \QCoh^G(\Bcal) \ar[d]^{\text{For}} \\ 
\Dcal^+\QCoh(\Acal) \ar[r]^{Rf_*} & \Dcal \QCoh(\Bcal)} \]
\end{enumerate}
\end{lemma}

\begin{lemma}\cite[Prop. 5.2.1]{BR} \label{QuasiIso}
Let $H$ be an algebraic group (not necessarily reductive) and $f : (X,\Acal) \to (X,\Bcal)$ a $H$-equivariant quasi-isomorphism of complex algebraic $H$-varieties satisfying the conditions above. Then the pull-back and push-forward functors induce equivalences of categories
\[ \Dcal \QCoh^H(\Acal) \iso \Dcal \QCoh^H(\Bcal).  \]
The equivalence restricts to an equivalence 
\[\Dcal \Coh^H(\Acal) \iso \Dcal \Coh^H(\Bcal).\]
\end{lemma}

\section{Equivariant linear Koszul duality} \label{SecEquiKoszulDual}
In the paper \cite{MR3} Mirkovi\'{c} and Riche extend the linear Koszul duality from \cite{MR1} and \cite{MR2} to the equivariant setting. In this section we recall their construction. Consider a complex algebraic variety $X$ with an action of a reductive algebraic group $G$. Again we assume that property \eqref{Assump} is satisfied. Consider a two term complex of locally free $G$-equivariant $\Ocal_X$-modules of finite rank.
\[ \Xcal := (\cdots 0 \to \Vcal \stackrel{f}{\longrightarrow} \Wcal \to 0 \cdots).\]
Here $\Vcal$ sits in homological degree -1 and $\Wcal$ is in homological degree 0. We consider it as a complex of graded $\Ocal_X$-modules with both $\Vcal$ and $\Wcal$ siting in internal degree 2. We define the graded symmetric algebra $\symcal_{\Ocal_X}(\Xcal)$ to be the sheaf tensor algebra of $\Xcal$ modulo the graded commutation relations $a \otimes b=(-1)^{\deg_h(a) \deg_h(b)} b \otimes a$, where $\deg_h$ is the homological degree. More explicitly $\symcal_{\Ocal_X}(\Xcal)$ is the bi-graded complex for which the term in homological degree $k$ and internal degree $2k+2n$ is
\[ \symcal_{\Ocal_X}(\Xcal)^k_{2k+2n}=\Lambda^k \Vcal \otimes_{\Ocal_X} \Sym^n(\Wcal). \]
The differential is given by
\[ d(v_1 \wedge \cdots \wedge v_n \otimes w)=\sum_{i=1}^n (-1)^i v_1 \wedge \cdots \wedge \hat{v_i} \wedge \cdots \wedge v_n \otimes f(v_i) w. \]
For a bi-graded sheaf of $\Ocal_X$-modules $\Mcal$ we denote by $\Mcal^\vee$ the bi-graded $\Ocal_X$-module with $(\Mcal^\vee)_j^i = \Hcal om_{\Ocal_X}(\Mcal_{-j}^{-i}, \Ocal_X)$. The dual complex is defined as
\[ \Ycal:=(\cdots 0 \to \Wcal^\vee \stackrel{-f^\vee}{\longrightarrow} \Vcal^\vee \to 0 \cdots ), \]
where $\Wcal^\vee$ sits in bi-degree (-1,-2) and $\Vcal^\vee$ sits in bi-degree (0,-2). A shift in homological degree is denoted by $[\;]$ and shift in internal degree is denoted by $(\;)$. We introduce the following notation
\[ \Tcal :=\symcal_{\Ocal_X}(\Xcal), \qquad \Rcal :=\symcal_{\Ocal_X}(\Ycal), \qquad \Scal :=\symcal_{\Ocal_X}(\Ycal[-2]). \]
Mirkovi\'{c} and Riche proved the following theorem known as equivariant linear Koszul duality.

\begin{thm} \cite[Theorem 3.1]{MR3} \label{EquiKoszulDual}
There is an equivalence of triangulated categories
\[ \kappa : \Dcal \Coh^{G \times \C^*}(\Tcal) \isomap \Dcal \Coh^{G \times \C^*}(\Rcal)^{op}, \]
satisfying $\kappa(\Mcal[n]( m ))=\kappa(\Mcal)[-n+m]( -m)$.
\end{thm}

\begin{rem}
In \cite{MR3} the theorem is stated in less generality. However, the corresponding statement in the non-equivariant setting \cite[Thm 1.7.1 and section 1.8]{MR2} is stated in this generality and the proof in \cite{MR3} shows that this non-equivariant equivalence can be lifted to the equivariant setting.
\end{rem}

We now recall their construction of the functor $\kappa$. For a dg-algebra $\Acal$ let $\Ccal_- \QCoh^{G \times \C^*}(\Acal)$ (resp. $\Ccal_+ \QCoh^{G \times \C^*}(\Acal)$) denote the full subcategory of $\Ccal \QCoh^{G \times \C^*}\linebreak[1](\Acal)$ consisting of objects whose internal degree is bounded above (resp. below) uniformly in the homological degree. The associated derived category is denoted by $\Dcal_- \QCoh^{G \times \C^*}(\Acal)$ (resp. $\Dcal_+ \QCoh^{G \times \C^*}(\Acal)$). The functor $\kappa$ is a restriction of a functor
\[ \kappa : \Dcal^{bc}_+\QCoh^{G\times \C^*}(\Tcal) \isomap \Dcal^{bc}_-\QCoh^{G \times \C^*}(\Rcal)^{op}. \]
This functor is the composition of three functors. The first is the functor
\[ \mathscr{A} : \Ccal \QCoh^{G \times \C^*}(\Tcal) \to \Ccal \QCoh^{G \times \C^*}(\Scal). \]
As a bi-graded equivariant $\Ocal_X$-module $\mathscr{A}(\Mcal)=\Scal \otimes_{\Ocal_X} \Mcal$. The $\Scal$-action is induced by the left multiplication of $\Scal$ on itself. The differential is the sum of two terms $d_1$ and $d_2$. The term $d_1$ is the natural differential on the tensor product
\[ d_1(s \otimes m)=d_{\Scal}(s) \otimes m +(-1)^{|s|}s \otimes d_{\Mcal}(m). \]
The term $d_2$ is the composition of the following morphisms. First the morphism
\[ \rho : \Scal \otimes_{\Ocal_X} \Mcal \to \Scal \otimes_{\Ocal_X} \Mcal, \qquad s \otimes m \mapsto (-1)^{|s|} s \otimes m. \]
The second morphism 
\[ \psi : \Scal \otimes_{\Ocal_X} \Mcal \to \Scal \otimes \Xcal^\vee \otimes \Xcal \otimes \Mcal \]
is induced by the natural morphism $i :\Ocal_X \to \Ecal nd(\Xcal) \iso \Xcal^\vee \otimes \Xcal$. The last map is the morphism
\[ \Psi :\Scal \otimes \Xcal^\vee \otimes \Xcal \otimes \Mcal \to \Scal \otimes \Mcal \]
induced by the right multiplication $\Scal \otimes_{\Ocal_X} \Xcal^\vee \to \Scal$ and the action $\Xcal \otimes_{\Ocal_X} \Mcal \to \Mcal$. The term $d_2$ is defined as $d_2 = \Psi \circ \psi \circ \rho$. Locally, choosing a basis $\{x_\alpha\}$ of $\Xcal$ and the dual basis $\{x_\alpha^*\}$ of $\Xcal^\vee$ it can be written as
\[ d_2(s \otimes m)=(-1)^{|s|}\sum_{\alpha} s x_{\alpha}^* \otimes x_\alpha \cdot m. \]
This data defines a $\Scal$-dg-module structure on $\mathscr{A}(\Mcal)$. Part of the proof of \cite[Thm. 3.1]{MR3} is showing that $\mathscr{A}$ induces an equivalence of categories
\[ \bar{\mathscr{A}}: \Dcal^{bc}_-\QCoh^{G \times \C^*}(\Tcal) \isomap \Dcal^{bc}_-\QCoh^{G \times \C^*}(\Scal). \]

By \cite[Example 2.16]{Bez1} under the assumption \ref{Assump} there exists an object $\Omega \in \Dcal^b \Coh^G(X)$ whose image under the forgetful functor For$:\Dcal^b \Coh^G(X) \to \Dcal^b \Coh(X)$ is a dualizing object in $\Dcal^b \Coh(X)$. Let $\mathcal{I}_\Omega$ be a bounded below complex of injective objects of $\QCoh^G(X)$ whose image in the derived category $\Dcal^+ \QCoh(X)$ is $\Omega$. It defines a functor on the category of complexes of all equivariant sheaves on $X$.
\[ \Rcal \Hcal om_{\Ocal_X}(-, \mathcal{I}_\Omega) : \Ccal ( \text{Sh}^G(X) )\to \Ccal( \text{Sh}^G(X))^{op}.\]
In \cite[Lemma 2.3]{MR3} it is proved that this functor is exact and that the induced functor on derived categories restricts to a functor
\[D_\Omega^X : D^b\Coh^G(X) \to D^b\Coh^G(X)^{op}. \]
Let $\tilde{\Ccal}(\Tcal-\module^{G \times \C^*})$ denote the category of all sheaves of $G \times \C^*$-equivariant $\Tcal$ dg-modules on $X$. It's derived category is denoted by $\tilde{\Dcal}(\Tcal-\module^{G \times \C^*})$. Consider the functor
\[ \tilde{D}_\Omega : \tilde{\Ccal}(\Tcal-\module^{G \times \C^*}) \to \tilde{\Ccal}(\Tcal-\module^{G \times \C^*})^{op}, \]
which sends $\Mcal \in \tilde{\Ccal}(\Tcal-\module^{G \times \C^*})$ to the dg-module whose underlying $G \times \mathbb{G}_m$-equivariant $\Ocal_X$-dg-module is $\Hcal om(\Mcal, \mathcal{I}_\Omega)$, with $\Tcal$-action defined by
\[ (t \cdot \phi)(m)=(-1)^{|t| \cdot |\phi |}\phi(t \cdot m), \qquad t \in \Tcal, m \in \Mcal. \]
In proposition 2.6 Mirkovi\'{c} and Riche prove that the induced functor restricts to an equivalence
\[ D_\Omega : \Dcal^{bc}\QCoh^{G \times \C^*}(\Tcal) \isomap \Dcal^{bc}\QCoh^{G \times \C^*}(\Tcal)^{op}.\]

The last functor is the regrading functor
\[ \xi : \Ccal \QCoh^{G \times \C^*}(\Scal) \to \Ccal \QCoh^{G \times \C^*}(\Rcal) \]
sending $\Mcal \in \Ccal \QCoh^{G \times \C^*}(\Scal)$ to the $\Rcal$-dg-module with $(i,j)$ component $\xi(\Mcal)=\Mcal_j^{i-j}$. If one forgets the grading then $\Scal$ and $\Rcal$ coincides and so does $\Mcal$ and $\xi(\Mcal)$. The $\Rcal$-action on the differential on $\xi(\Mcal)$ is the same as the $\Scal$-action on the differential of $\Mcal$. This is an equivalence of categories. The functor $\kappa$ from theorem \ref{EquiKoszulDual} is the restriction of the composition $\xi \circ \bar{\mathscr{A}} \circ D_\Omega$.

\subsection{Extension to an arbitrary linear algebraic group}
We want to be able to work with equivariance with respect to a Borel subgroup. Thus, we need to extend linear Koszul duality to work with a not necessarily reductive linear algebraic group $H$ sitting inside a reductive group $G$. Let $X$ be a $H$-variety. Consider the variety $\tilde{X} := \Ind_H^G(X) = \tfrac{G \times X}{H}$. The projection and quotient morphisms
\[ X \stackrel{\pr}{\longleftarrow} G \times X \stackrel{\pi}{\longrightarrow} \tilde{X}.  \]
induce equivalences of categories
\begin{align} 
&A:= \pr^* : \QCoh^H(X) \isomap \QCoh^{G \times H}(G \times X),\\
& B:=\pi^* : \QCoh^G(\tilde{X}) \isomap \QCoh^{G \times H}(G \times X).
\end{align}

\begin{lemma}
The functors $A$ and $B$ are monoidal.
\end{lemma}
\begin{proof}
Let $\Delta : X \to X \times X$ be the diagonal embedding. By definition the monoidal action on $\QCoh^H(X)$ is given by
\[ M \otimes_{\Ocal_X} N := \Delta^* \Res_{H_\Delta}^{H\times H}(M \boxtimes N). \]
Consider the commutative diagram
\[ \xymatrix{X \ar[r]^-\Delta & X \times X \\ G \times X \ar[u]^\pr \ar[r]^-{\Delta_G} & G \times X \times G \times X \ar[u]^{\pr_2}} \]
Using this we calculate
\begin{align}
A(M) \otimes_{\Ocal_{G \times X}} A(N) & =\Delta_G^* \Res^{G \times H \times G \times H}_{(G \times H)_\Delta}(\pr^* M \boxtimes \pr^*N)\\ 
& \iso \Delta_G^* \Res^{G \times H \times G \times H}_{(G \times H)_\Delta}\pr^*_2(M \boxtimes N)\\
& \iso \Delta_G^* \pr_2^* \Res_{H_\Delta}^{H \times H}(M \boxtimes N)\\
& \iso \pr^* \Delta^* \Res_{H_\Delta}^{H \times H}(M \boxtimes N)\\
&=A(M \otimes_{\Ocal_X} N).
\end{align}
Thus, $A$ is monoidal. For $B$ we have the following diagram
\[ \xymatrix{\tilde{X} \ar[r]^-{\tilde{\Delta}} & \tilde{X} \times \tilde{X} \\ G \times X \ar[u]^\pi \ar[r]^-{\Delta_G} & G \times X \times G \times X \ar[u]^{\pi_2}} \]
This gives
\begin{align}
B(M) \otimes_{\Ocal_{G \times X}} B(N) &= \Delta_G^* \Res_{(G \times H)_\Delta}^{G \times H \times G \times H}(B(M) \boxtimes B(N))\\
&\iso \Delta_G^* \pi_2^* \Res_{G_\Delta}^{G \times G}(M \boxtimes N)\\
&\iso \pi^* \tilde{\Delta}^* \Res_{G_\Delta}^{G \times G}(M \boxtimes N)\\
&=B(M \otimes_{\Ocal_{\tilde{X}}} N).
\end{align}
Hence, both functors are monoidal.
\end{proof}

By the lemma we have a monoidal equivalence of categories
\[ B^{-1} A : \QCoh^H(X) \isomap \QCoh^G(\tilde{X}). \]
Consider a complex of $H$-equivariant vector bundles
\[ \Xcal := (\cdots \to 0 \to \Vcal \to \Wcal \to 0 \to \cdots). \]
Applying $B^{-1}A$ we get a new complex
\[ \tilde{\Xcal} :=B^{-1}A(\Xcal) =(\cdots \to 0 \to B^{-1}A(\Vcal) \to B^{-1}A(\Wcal) \to 0 \to \cdots). \]
Notice that
\begin{align} 
\Hcal om_{\Ocal_{\tilde{X}}}(B^{-1}A(\Mcal), \Ocal_{\tilde{X}}) &\iso \Hcal om_{\Ocal_{\tilde{X}}}(B^{-1}A(\Mcal), B^{-1}A(\Ocal_{X})) \\
&\iso B^{-1}A ( \Hcal om_{\Ocal_X}(\Mcal, \Ocal_X)).
\end{align}
Thus, $\widetilde{\Xcal^\vee} \iso (\tilde{\Xcal})^\vee$.

In the construction of $\symcal_{\Ocal_X}(\Xcal)$ we only used the monoidal structure on $\QCoh^H(X)$. Since $B^{-1}A$ is monoidal we get
\[ B^{-1}A(\symcal_{\Ocal_X}(\Xcal)) \iso \symcal_{\Ocal_{\tilde{X}}}(\tilde{\Xcal}). \]
Let $\Mcal$ be a dg-module over $\symcal_{\Ocal_X}(\Xcal)$. I.e. there is a collection of linear maps $\symcal_{\Ocal_X}^n(\Xcal) \otimes_{\Ocal_X} \Mcal \to \Mcal$ respecting the differentials. $B^{-1}A$ respects these maps so $B^{-1}A(\Mcal)$ is a dg-module over $\symcal_{\Ocal_{\tilde{X}}}(\tilde{\Xcal})$. Thus, we have proved

\begin{prop}
There is a natural equivalence of dg-categories
\[\Ccal \QCoh^H(\symcal_{\Ocal_X}(\Xcal)) \iso \Ccal\QCoh^G(\symcal_{\Ocal_{\tilde{X}}}(\tilde{\Xcal})).\]
\end{prop}

The functor sends quasi-isomorphisms to quasi-isomorphisms so it descends to the derived category. Using this equivalence we obtain the desired version of linear Koszul duality.

\begin{thm}
Let $G$ be a complex reductive group acting on a variety $X$ satisfying condition \ref{Assump}. Let $H$ be a closed subgroup of $G$ and define $\Tcal$ and $\Rcal$ as in the previous section. Then there is an equivalence of triangulated categories
\[ \kappa : \Dcal \Coh^{H \times \C^*}(\Tcal) \isomap \Dcal \Coh^{H \times \C^*}(\Rcal)^{op}, \]
satisfying $\kappa(\Mcal[n]( m ))=\kappa(\Mcal)[-n+m]( -m)$.
\end{thm}

\section{Derived category of equivariant DG-modules for $G$-schemes} \label{Aside}

In this section we extend the construction in \cite{Isik} to the equivariant setting. Let $X$ be a smooth complex algebraic variety with an action of a reductive algebraic group $G$. In particular, $X$ is Noetherian, separated and regular. Then $X$ has an ample family of $G$-equivariant line bundles and property \eqref{Assump} is satisfied (see \cite[Remark 1.5.4]{VV}). Let $\pi : E \to X$ be a $G$-equivariant vector bundle of rank $n$. Denote the sheaf of $G$-equivariant sections of the bundle by $\Ecal$ and let $s \in H^0(X, \Ecal)$ be a $G$-equivariant regular section. The zero scheme of $s$ is denoted by $Y$. In order to use linear Koszul duality we need to introduce an additional $\Z$-grading or equivalently a $\C^*$-action. Consider $\Ocal_Y[t,t^{-1}]$ as a bi-graded dg-algebra sitting in homological degree 0 with zero differential and $t$ a formal variable sitting in internal degree -2.

\begin{prop} \label{DAcal}
There is an equivalence of categories
\[ \Dcal \Coh^{G \times \C^*}(\Ocal_Y[t,t^{-1}]) \iso D^b \Coh^G(Y)\]
\end{prop}
\begin{proof}
The pull-back along the projection $Y \times \C^* \to Y$ is an equivalence of categories $\Coh^G(Y) \iso \Coh^{G \times \C^*}(Y \times \C^*)$. By remark \ref{OX}
\[ D^b\Coh^{G \times \C^*}(Y \times \C^*) \iso \Dcal \Coh^{G \times \C^*}( \Ocal_{Y \times \C^*}). \]
Notice that $\Ocal_{Y \times \C^*} \iso \Ocal_Y[t,t^{-1}]$.
\end{proof}

By lemma \ref{QuasiIso} we may replace $\Ocal_Y[t,t^{-1}]$ by a quasi-isomorphic dg-algebra siting in non-positive homological degrees which fits into the setting of linear Koszul duality. When $Y$ is the zero locus of a regular section $s \in H^0(X,\Ecal)$ the sheaf $\Ocal_Y$ has an equivariant Koszul resolution
\[ 0 \to \Lambda^n \Ecal^\vee  \to \cdots \to \Lambda^2 \Ecal^\vee \to \Ecal^\vee \to \Ocal_X \to \Ocal_Y \to 0 \]
with differential given by $d(f)=f(s)$ and extended by Leibnitz rule. Using shifted copies of this resolution in each internal degree we get a bi-complex with is a resolution of $\Ocal_Y[t,t^{-1}]$.
\[ \xymatrix{\cdots \ar[r] & \Lambda^3 \Ecal^\vee t^{-1} \ar[r] & \Lambda^2 \Ecal^\vee \ar[r] & \Ecal^\vee t \ar[r] & \Ocal_X t^2 & i=-4\\
\cdots \ar[r] & \Lambda^3 \Ecal^\vee t^{-2} \ar[r] & \Lambda^2 \Ecal^\vee t^{-1} \ar[r] & \Ecal^\vee \ar[r] & \Ocal_X t & i=-2\\
\cdots \ar[r] & \Lambda^3 \Ecal^\vee t^{-3} \ar[r] & \Lambda^2 \Ecal^\vee t^{-2} \ar[r] & \Ecal^\vee t^{-1} \ar[r] & \Ocal_X & i=0\\
\cdots \ar[r] & \Lambda^3 \Ecal^\vee t^{-4} \ar[r] & \Lambda^2 \Ecal^\vee t^{-3} \ar[r] & \Ecal^\vee t^{-2} \ar[r] & \Ocal_X t^{-1} & i=2\\
\cdots \ar[r] & \Lambda^3 \Ecal^\vee t^{-5} \ar[r] & \Lambda^2 \Ecal^\vee t^{-4} \ar[r] & \Ecal^\vee t^{-3} \ar[r] & \Ocal_X t^{-2} & i=4\\} \]
We denote this bi-complex by $\Acal_{X \times \C^*}$. By construction $H(\Acal_{X \times \C^*})=\Ocal_Y[t,t^{-1}]$ and the morphism
\[ \psi : \Acal_{X \times \C^*} \to \Ocal_Y[t,t^{-1}],\]
which takes $t^k f$ to $t^k f|_Y$ for $f \in \Ocal_X$ and everything else to zero, is a quasi-isomorphism. Thus, we have shown that

\begin{prop} \label{AcalCohY}
There is an equivalence of categories
\[ \Dcal \Coh^{G \times \C^*}(\Acal_{X \times \C^*}) \iso \Dcal \Coh^{G \times \C^*}(\Ocal_Y[t,t^{-1}]). \]
\end{prop}

Consider the following bi-graded complex with $\Ecal^\vee$ in degree (-1,-2), $\Ocal_X$ in degree (0,0) and $t$ in degree $(0,-2)$.
\[ \Acal_{X \times \C} := \bigwedge \Ecal^\vee \otimes_{\Ocal_X} \Ocal_X[t]\]
with differential $d(f)=tf(s)$ and extended by Leibnitz. Observe that 
\[\Acal_{X \times \C^*}=\Acal_{X \times \C} \otimes_{\Ocal_X} \Ocal_X[t^{-1}]. \]
The bi-complex $\Acal_{X \times \C}$ fits into the setting of linear Koszul duality as it can be written in the form
\[ \Acal_{X \times \C}=\symcal_{\Ocal_X} \bigl(0 \to \Ecal^\vee \stackrel{-s^\vee}{\longrightarrow} t \Ocal_X \to 0 \bigr) \]

\begin{defin}
Let $(X,\Acal)$ be a $G \times \C^*$-equivariant dg-scheme.
\begin{enumerate}
\item The full subcategory of $\Dcal \Coh^{G \times \C^*}(\Acal)$ whose objects are locally in $\langle \Acal(i) \rangle_{i \in \Z}$, i.e. is quasi-isomorphic to a bounded complex of free $\Acal$-modules of finite rank, is denoted by $\Perf(\Acal)$. Such complexes are called perfect.
\item The full subcategory of $\Dcal \Coh^{G \times \C^*}(\Acal)$ whose objects are locally in $\langle \Ocal_X(i) \rangle_{i \in \Z}$ is denoted by $\Dcal_X \Coh^{G \times \C^*}(\Acal)$. We say that these modules are supported on $X$.
\end{enumerate}
\end{defin}

\begin{lemma}
Let $G$ be a complex reductive algebraic group acting on $X$ such that assumption \ref{Assump} is satisfied. Let $\Rcal$ be a sheaf of dg-algebras sitting in non-positive homological and non-positive internal degrees with $\Rcal_0^0=\Ocal_X$ and $H^0(\Rcal)_0=\Ocal_X$. If $\Mcal$ is a coherent module over $\Rcal$ and $H(\Mcal)$ is coherent when considered as a module over $\Ocal_X$, then $\Mcal$ is locally in $\langle \Ocal_X(i) \rangle_{i \in \Z}$.
\end{lemma}
\begin{proof}
In the non-equivariant setting this is \cite[Lemma 3.3]{Isik}. Under the assumption that \eqref{Assump} is satisfied the proof extends to the equivariant setting. We recall it here.
The property is local so we may assume that $X$ is affine. Assume that $H(\Mcal)$ is coherent as a $\Ocal_X$-module. In particular the cohomology of $\Mcal$ is bounded above and below and there are only finitely many pairs $(i,j)$ such that $H^i(\Mcal)_j\neq 0$. The proof is by induction on the number of such pairs. Acyclic modules are in $\langle \Ocal_X(i) \rangle_{i \in \Z}$ so the start is clear.

Let $n$ be the lowest degree such that $H^n(\Mcal)\neq 0$. Then $\Mcal$ is quasi-isomorphic to the truncated complex 
\[ \tau_{\geq n} \Mcal = \quad \cdots \to 0 \to \coker d_\Mcal^{n-1} \to \Mcal^{n+1} \to \Mcal^{n+2} \to \cdots \]
The complex $\tau_{\geq n} \Mcal$ is also a $\Rcal$-module since $\Rcal$ sits in non-positive homological degrees. Let $\Fcal$ be the kernel of the morphism
\[ d^n: \coker d_\Mcal^{n-1} \to \Mcal^{n+1}. \]
The assumption on $\Rcal$ means that $\Fcal$ is a $\Rcal$-submodule of $\tau_{\geq n} \Mcal$. Notice that $\Fcal \iso H^n(\Mcal)$, so $\Fcal$ is coherent as an $\Ocal_X$-module.

Let $m$ be the lowest internal degree such that $\Fcal_m \iso H^n(\Mcal)_m \neq 0$. Since $\Rcal$ sits in non-positive both homological and internal degree $\Rcal_j^i$ acts by zero on the coherent $\Ocal_X$-module $\Fcal_m$ for $(i,j) \neq (0,0)$. Thus, $\Fcal_m$ which is concentrated in degree $(n,m)$, is also an $\Rcal$-module.

Since $X$ is smooth there exist a finite $G$-equivariant free resolution of $\Fcal_m$. Hence, $\Fcal_m$ is quasi-isomorphic to a complex of $\Ocal_X$-modules of the form
\[ 0 \to \Ocal_X^{\oplus r_k}(m) \to \cdots \to \Ocal_X^{\oplus r_2}(m) \to \Ocal_X^{\oplus r_1}(m) \to 0.\]
Since $\Rcal_j^i$ acts by zero on $\Fcal_m$ for $(i,j) \neq (0,0)$ and $\Rcal_0^0=\Ocal_X$ this complex is also quasi-isomorphic to $\Fcal_m$ as a $\Rcal$-module with $\Rcal$ acting trivially except for the $(0,0)$ piece. Hence, $\Fcal_m$ represents an object in $\langle \Ocal_X(i) \rangle_{i \in \Z}$. The cone of the inclusion $\Fcal_m \hookrightarrow \tau_{\geq n} \Mcal$ has the same cohomology as $\Mcal$ except for the piece in degree $(n,m)$ which is zero. By induction the cone is in $\langle \Ocal_X(i) \rangle_{i \in \Z}$ so $\tau_{\geq n} \Mcal$ and consequently $\Mcal$ are in $\langle \Ocal_X(i) \rangle_{i \in \Z}$.
\end{proof}

\begin{prop}\label{AcalQuotient}
There is an equivalence of categories
\[ \frac{\Dcal \Coh^{G \times \C^*}(\Acal_{X \times \C})}{\Dcal_X \Coh^{G \times \C^*}(\Acal_{X \times \C})} \iso \Dcal \Coh^{G \times \C^*}(\Acal_{X \times \C^*}). \]
\end{prop}
\begin{proof}
In the non-equivariant setting this is \cite[Prop. 3.2]{Isik}. In the proof Isik uses the inclusion morphism $\phi : \Acal_{X \times \C} \to \Acal_{X \times \C}[t^{-1}]=\Acal_{X \times \C^*}$ to construct two functors
\begin{gather}
\phi_* : \Dcal \QCoh(\Acal_{X \times \C}) \to \Dcal \QCoh(\Acal_{X \times \C^*}), \\\
\Mcal \mapsto \Acal_{X \times \C^*} \otimes_{\Acal_{X \times \C}} \Mcal \iso \Ocal_X[t,t^{-1}] \otimes_{\Ocal_X[t]} \Mcal,\\
\phi^* : \Dcal \QCoh(\Acal_{X \times \C^*}) \to \Dcal \QCoh(\Acal_{X \times \C}), \\
 \Ncal \mapsto \Ncal_{\leq 0}.
\end{gather}

He proves that $\phi_*$ factors through $\Dcal_X \Coh(\Acal_{X \times \C})$ and that the functors induce mutually inverse equivalences of categories
\[\frac{\Dcal \Coh(\Acal_{X \times \C})}{\Dcal_X \Coh(\Acal_{X \times \C})} \xymatrix@=3em{ \ar@<0.7ex>[r]^-{\phi_*}_-\sim & \Dcal \Coh(\Acal_{X \times \C^*}) \ar@<0.7ex>[l]^-{\phi^*} } \]
Both functors naturally extend to the equivariant setting so we get functors
\[\frac{\Dcal \Coh^{G \times \C^*}(\Acal_{X \times \C})}{\Dcal_X \Coh^{G \times \C^*}(\Acal_{X \times \C})} \xymatrix@=3em{ \ar@<0.7ex>[r]^-{\phi_*} & \Dcal \Coh^{G \times \C^*}(\Acal_{X \times \C^*}) \ar@<0.7ex>[l]^-{\phi^*} } \]
As in the non-equivariant setting we need to prove that the natural transformations 
\[ \phi_* \circ \phi^* \to \id \]
given by
\[  \Acal_{X \times \C}[t^{-1}] \otimes_{\Acal_{X \times \C}} (\Ncal)_{\leq 0} \to \Ncal, \qquad a \otimes n \mapsto an. \]
and 
\[ \id \to \phi^* \circ \phi_*\]
given by
\begin{gather}
\Mcal \to (\Acal_{X \times \C}[t^{-1}] \otimes_{\Acal_{X \times \C}} \Mcal)_{\leq 0} \iso (k[t^{-1}] \otimes_k \Mcal)_{\leq 0}, \\
m \mapsto \begin{cases} 1 \otimes m  & \text{internal degree non-positive} \\ 0 & \text{else} \end{cases}
\end{gather}
are both isomorphisms.

Then first natural transformation $\phi_* \circ \phi^* \to \id$ is clearly surjective. Assume that we have a section $t^k \otimes n$ whose image is 0. Then $n=t^{-k} t^kn=0$, so the morphism is also injective. Thus, this natural transformation is an isomorphism

Let $\Jcal$ be the cone of the morphism $\Mcal \to \phi^* \phi_* \Mcal$ from the second natural transformation. We want to show that $\Jcal$ is supported on $X$. By the lemma it is enough to show that $H(\Jcal)$ is coherent over $\Ocal_X$. Consider the long exact sequence of sheaves of $\Ocal_X$-modules in cohomology
\[ \cdots \to H^i(\Mcal) \to H^i(\phi^* \phi_* \Mcal) \to H^i(\Jcal) \to H^{i+1}(\Mcal) \to H^{i+1}(\phi^* \phi_* \Mcal) \to \cdots \]
So we get short exact sequences
\[  0 \to \coker(\alpha_i) \to H^i(\Jcal) \to \ker(\alpha_{i+1}) \to 0,\]
where
\[ \alpha_i : H^i(\Mcal) \to H^i((\Ocal_X[t,t^{-1}] \otimes_{\Ocal_X[t]} \Mcal)_{\leq 0}) \]
is the induced map on cohomology. From the short exact sequence it follows that $H(\Jcal)$ is coherent over $\Ocal_X$ if $\coker(\alpha_i)$ and $\ker(\alpha_{i+1})$ are. Recall what the terms in $\Acal_{X \times \C}$ look like
\[ \xymatrix{\cdots \ar[r]  &0 \ar[r] & \Lambda^2 \Ecal^\vee \ar[r] & \Ecal^\vee t \ar[r] & \Ocal_X t^2 & i=-4\\
\cdots \ar[r]  & 0 \ar[r] & 0 \ar[r] & \Ecal^\vee \ar[r] & \Ocal_X t & i=-2\\
\cdots \ar[r] & 0 \ar[r] & 0 \ar[r]  & 0 \ar[r] & \Ocal_X & i=0} \]
In each degree we have a truncation of the resolution of $\Ocal_Y$. In degrees lower than $-n$ we have the full resolution so the only non-zero cohomology in lower degrees are of the form $t^k \Ocal_Y$. All individual terms in $H(\Acal_{X \times \C})$ are coherent over $\Ocal_X$ so the only way $\coker(\alpha_i)$ and $\ker(\alpha_{i+1})$ could fail to be coherent is if infinitely many powers of $t$ are required to generate them. However, everything in $\ker(\alpha_{i+1})$ sits in strictly positive internal degree and since the internal degree of $t$ is $-2$ this is not the case. Isik shows that $H(\Ocal_X[t,t^{-1}] \otimes_{\Ocal_X[t]} \Mcal) \iso \Ocal_X[t,t^{-1}] \otimes_{\Ocal_X[t]} H(\Mcal)$ so elements in $\coker(\alpha_{i})$ are represented by $t^{-k} \otimes m$ with $2k+\deg_i(m) \leq 0$, where $\deg_i(m)$ is the internal degree of $m$. Thus, $\coker(\alpha_i)$ is also coherent over $\Ocal_X$ and we get the result.
\end{proof}

The Koszul dual to $\Acal_{X \times \C}$ is the dg-algbera
\[ \Bcal :=\symcal( 0 \to \epsilon \Ocal_X \stackrel{s}{\longrightarrow} \Ecal \to 0). \]
Here $\epsilon$ is a formal variable with $\epsilon^2=0$ sitting in homological degree -1 and internal degree 2. This is just a convenient notation expressing the fact that $\Lambda^n \Ocal_X=0$ for $n \geq 2$. $\Ecal$ is in homological degree 0 and internal degree 2. As a complex $\Bcal$ is $\epsilon \Sym \Ecal \to \Sym \Ecal$ with differential $d_\Bcal(\epsilon f)=sf$.

The functor $\kappa$ from theorem \ref{EquiKoszulDual} gives an equivalence of categories
\[ \kappa : \Dcal \Coh^{G \times \C^*}(\Bcal) \isomap \Dcal \Coh^{G \times \C^*}(\Acal_{X \times \C})^{op}. \]

\begin{lemma}
The functor $\kappa$ restricts to an equivalence 
\[\operatorname{Perf}(\Bcal) \iso \Dcal_X \Coh^{G \times \C^*}(\Acal_{X \times \C})^{op}.\]
\end{lemma}
\begin{proof}
In the non-equivariant setting this is similar to proposition 3.1 in \cite{Isik}. However, Isik uses the Koszul duality from \cite{MR1}, which is slightly different from the linear Koszul duality from \cite{MR3} that we are using, so some additional arguments are needed. The functor $\kappa$ is defined locally so for any open subset $i : U \hookrightarrow X$ the following diagram is commutative
\[ \xymatrix{\Dcal \Coh^{G \times \C^*}(\Bcal) \ar[r]^-\kappa \ar[d]_{i^*} & \Dcal \Coh^{G \times \C^*}(\Acal_{X \times \C}) \ar[d]_{i^*} \\
\Dcal \Coh^{G \times \C^*}(\Bcal|_U) \ar[r]^-{\kappa|_U} & \Dcal \Coh^{G \times \C^*}(\Acal_{X \times \C}|_U)} \]
It suffices to show that for any open affine $U$ the functor $\kappa|_U$ takes modules of the form $\Bcal(i)$ to objects in $\langle \Ocal_X(i) \rangle_{i \in \Z}$. Theorem \ref{EquiKoszulDual} states that $\kappa(\Mcal[j](i))=\kappa(\Mcal)[-j+i](-i)$, so it is enough to prove the statement for $i=0$.

Recall that $\kappa$ is the composition of three functors $\xi$, $\bar{\mathscr{A}}$ and $D_\Omega$. Since $X$ is smooth the dualizing sheaf $\Omega$ is just top forms shifted by dimension. In particular, it is locally free of rank 1 so the functor $\Hcal om(-,\Omega)$ is exact and there is no need to take the injective resolution $I_\Omega$. That $\Omega$ is a line bundle implies that locally $D_\Omega(\Bcal) \iso \Hcal om(\Bcal,\Ocal_X)$. Applying $\bar{\Acal}$ on the opposite category corresponds to reversing the grading for all dg-modules. Thus, $\bar{\Acal} \circ D_\Omega(\Bcal)$ corresponds to $\bar{\Acal}(\Bcal^\vee)$, where $(\Bcal^\vee)_j^i=\Hcal om(\Bcal_{-j}^{-i},\Ocal_X)$. By \cite[Lemma 2.6.1]{MR1} the projection on the $(0,0)$-component $\bar{\Acal}(\Bcal^\vee) \to \Ocal_X$ is a quasi-isomorphism. Clearly, $\xi(\Ocal_X) =\Ocal_X$. This finishes the proof.
\end{proof}

\begin{defin}
The singularity category of a dg-algebra $\Acal$ is the Verdier quotient
\[ \Dcal_{sg}^{G \times \C^*}(\Acal) := \frac{\Dcal \Coh^{G \times \C^*}(\Acal)}{\Perf(\Acal)}. \]
\end{defin}

\begin{cor}
There is an equivalence of categories
\[ \Dcal^{G \times \C^*}_{sg}(\Bcal) \iso \frac{\Dcal \Coh^{G \times \C^*}(\Acal_{X \times \C})^{op}}{\Dcal_X \Coh^{G \times \C^*}(\Acal_{X \times \C})^{op}}. \]
\end{cor}

Recall the notation from the beginning of this section. Let $\pi^\vee : E^\vee \to X$ denote the dual vector bundle. This defines a pull-back section via the cartesian diagram
\[ \xymatrix{E^\vee \times_X E \ar[r] & E \\ E^\vee \ar[u]^{(\pi^\vee)^*s} \ar[r]^{\pi^\vee} & X \ar[u]^s} \]
Define the function $W$ to be the composition with the natural pairing
\[ W: E^\vee \to E^\vee \times_X E \stackrel{\langle \;, \: \rangle}{\longrightarrow} \C, \qquad a_x \mapsto \langle ((\pi^\vee)^*s)(a_x), a_x \rangle. \]

Set $Z:=W^{-1}(0)$. There is a short exact sequence
\[0 \to \epsilon \symcal \Ecal \stackrel{s}{\longrightarrow} \symcal \Ecal \longrightarrow \pi_* \Ocal_Z \to 0. \]
Thus, the map $\phi : \Bcal \to \pi_* \Ocal_Z$ sending $\symcal \Ecal$ to $\pi_* \Ocal_Z$ and $\epsilon$ to 0 is a quasi-isomorphism of sheaves of equivariant graded dg-algebras. Proposition \ref{QuasiIso} gives an equivalence
\[ \Dcal \Coh^{G \times \C^*}(\Bcal) \iso \Dcal \Coh^{G \times \C^*}(\pi_* \Ocal_Z). \]
The equivalence takes $\Bcal$ to $\pi_* \Ocal_Z$ so it descends to the singularity categories.

\begin{lemma}\label{BsingPi}
There is an equivalence of categories $\Dcal^{G \times \C^*}_{sg}(\Bcal) \iso \Dcal^{G \times \C^*}_{sg}(\pi_* \Ocal_Z)$.
\end{lemma}

\begin{lemma}\label{fPush}
Let $f:X \to Y$ be a $G$-equivariant affine morphism and set $\Ccal :=f_* \Ocal_X$. Then $f_*$ induces an equivalence of categories
\[ f_* : \QCoh^G(X) \isomap \QCoh^G(\Ccal). \]
Here $\QCoh^G(\Ccal)$ denotes $G$-equivariant quasi-coherent $\Ocal_Y$-modules with a $\Ccal$-action.
\end{lemma}
\begin{proof}
In the non-equivariant setting this is \cite[Exercise II.5.17]{Har1}. The functor $f_*$ lifts to the equivariant setting
\[ \xymatrix{\QCoh^G(X) \ar[r]^{f_*^G} \ar[d]^{\text{For}}  & \QCoh^G(\Ccal) \ar[d]^{\text{For}} \\
\QCoh(X) \ar[r]^{f_*}_\sim & \QCoh(\Ccal)}\]
Since $f_*$ is fully-faithful we have
\begin{align}
\Hom_\Ccal^G(f_*^G(\Ncal), f_*^G(\Mcal)) &=\left(\Hom_\Ccal(f_*(\Ncal), f_*(\Mcal))\right)^G\\
&=\left(\Hom_X(\Ncal,\Mcal) \right)^G\\
&=\Hom_X^G(\Ncal,\Mcal).
\end{align}
Thus, $f_*^G$ is fully-faithful. We need to show that $f_*^G$ is essentially surjective. Let $(\Mcal, \phi) \in \QCoh^G(\Ccal)$ where $\Mcal \in \QCoh(\Ccal)$ and $\phi : ac^*_Y(\Mcal) \isomap p_Y^*(\Mcal)$. Here $ac_Y, p_Y : G \times Y \to Y$ is the action and projection morphism, respectively. Use similar notation for $X$. Since $f_*$ is essentially surjective forgetting the equivariance $\Mcal \iso f_*(\Ncal)$ for some $\Ncal \in \QCoh(X)$. Using base change we get isomorphisms
\[ \xymatrix@=3em{(1 \times f_*)ac_X^*(\Ncal) \iso ac_Y^* f_*(\Ncal) \ar[r]_\sim^\phi & p_Y^* f_*(\Ncal) \iso (1 \times f_*)p_X^*(\Ncal)} \]
Since $1 \times f_*$ is an equivalence of categories this induces an isomorphism $\psi :=(1 \times f_*)^{-1} \phi : ac_X^*(\Ncal) \isomap p_X^*(\Ncal)$. Thus, $\Ncal$ lifts to $\QCoh^G(X)$ so $f_*^G$ is essentially surjective.
\end{proof}

The morphism $\pi : Z \to X$ is affine so the lemma gives an equivalence
\[ D^b\Coh^{G \times \C^*}(Z) \iso \Dcal \Coh^{G \times \C^*}(\pi_*\Ocal_Z). \]
The equivalence descends to the singularity categories
\[\Dcal^{G \times \C^*}_{sg}(\pi_* \Ocal_Z) \iso D^{G \times \C^*}_{sg}(Z).\]

Combining all these equivalences of categories we can now prove the equivariant version of the main theorem in \cite{Isik}.

\begin{thm} \label{ZCoh}There is an equivalence of categories $D^{G \times \C^*}_{sg}(Z) \iso D^b(\Coh^G(Y))$.
\end{thm}
\begin{proof}
We already proved the following series of equivalences
\begin{align}
D_{sg}^{G \times \C^*}(Z) &\iso \Dcal_{sg}^{G \times \C^*}(\pi_* \Ocal_Z)\\
& \iso \Dcal_{sg}^{G \times \C^*}(\Bcal)\\
& \iso  \frac{\Dcal \Coh^{G \times \C^*}(\Acal_{X \times \C})^{op}}{\Dcal_X \Coh^{G \times \C^*}(\Acal_{X \times \C})^{op}}\\
& \iso \Dcal \Coh^{G \times \C^*}(\Acal_{X \times \C^*})^{op}\\
& \iso  \Dcal \Coh^{G \times \C^*}((\Ocal_Y[t,t^{-1}])^{op}\\
& \iso D^b (\Coh^G(Y))^{op}.
\end{align}
Since $X$ has an ample family of $G$-equivariant line bundles so does the closed subvariety $Y$ by restriction. Thus, $Y$ satisfies property \eqref{Assump} so a functor similar to $D_\Omega^X$ from section \ref{SecEquiKoszulDual} can also be constructed for $Y$. This gives an equivalence $D^b(\Coh^G(Y))^{op} \iso D^b(\Coh^G(Y))$.
\end{proof}

\section{Equivariant matrix factorizations and singularity category}

\subsection{Definitions} \label{DefMatFact1}
Let $X$ be an algebraic stack and $W \in H^0(X,\C)$ a section. The section $W$ is called the potential.

\begin{defin}
\begin{enumerate}
\item A matrix factorization $\bar{E}=(E_\bullet, \delta_\bullet)$ of $W$ on $X$ consists of a pair of vector bundles, i.e. locally free sheaves of finite rank, $E_0,E_1$ on $X$ together with homomorphisms
\[ \delta_1 : E_1 \to E_0 \quad \text{and} \quad \delta_0 : E_0 \to E_1 \]
such that $\delta_1 \delta_0=W \cdot \id =\delta_0 \delta_1$.

\item The dg-category of matrix factorizations is defined in the following way. Let $\bar{E}, \bar{F}$ be matrix factorizations. Then $\Hcal om_{\MF}(\bar{E}, \bar{F})$ is the $\Z$-graded complex
\begin{gather}
\Hcal om_{\MF}(\bar{E}, \bar{F})^{2n}:=\Hom(E_0, F_0) \oplus \Hom(E_1, F_1),\\
\Hcal om_{\MF}(\bar{E}, \bar{F})^{2n+1}:=\Hom(E_0, F_1) \oplus \Hom(E_1, F_0).
\end{gather}
with differential
\[ df:=\delta_F \circ f-(-1)^{|f|} f \circ \delta_E. \]
\end{enumerate}
\end{defin}

\begin{rem}
Matrix factorizations can be defined for a general line bundle over $X$ (see \cite{PV}) but for the application we have in mind we only need $L=\C$.
\end{rem}

Let $G$ be a linear algebraic group acting on $X$ and assume that $W$ is invariant with respect to the action. Then we define $G$-equivariant matrix factorizations in the following way.

\begin{defin}
A matrix factorization $\bar{E}=(E_\bullet, \delta_\bullet)$ of $W$ on $X$ is $G$-equivariant if $(E_0,E_1)$ are $G$-equivariant vector bundles and $(\delta_0,\delta_1)$ are $G$-invariant. We define $\Hcal om_{\MF_G}(\bar{E}, \bar{F})$ to be the complex
\begin{gather}
\Hcal om_{\MF_G}(\bar{E}, \bar{F})^{2n}:=\Hom(E_0, F_0)^G \oplus \Hom(E_1, F_1)^G,\\
\Hcal om_{\MF_G}(\bar{E}, \bar{F})^{2n+1}:=\Hom(E_0, F_1)^G \oplus \Hom(E_1, F_0)^G.
\end{gather}
The differential is the same as for non-equivariant matrix factorizations.
\end{defin}

Denote the corresponding homotopy categories by $\text{HMF}(X,W) :=H^0(\MF(X,W))$ and $\text{HMF}_G(X,W) :=H^0(\MF_G(X,W))$.

\begin{rem} \label{MFQuotient}
Let $\bar{W}$ denote the induced potential $X/G \to \C$. Then the dg-categories $\MF_G(X,W)$ and $\MF(X/G, \bar{W})$ are equivalent.
\end{rem}

The category HMF$(X,W)$ is a triangulated category. Consider the triangulated subcategory LHZ$(X,W)$ consisting of matrix factorizations $\bar{E}$ that are locally contractible (i.e. there exists an open covering $U_i$ of $X$ in the smooth topology such that $\bar{E}|_{U_i}=0$ in HMF$(U_i, W|_{U_i})$).

\begin{defin}
For a stack $X$ we define the derived category of matrix factorizations by
\[ \operatorname{DMF}(X,W):=\operatorname{HMF}(X,W)/ \operatorname{LHZ}(X,W). \]
\end{defin}

\subsection{Connection with singularity categories}

In this subsection we recall the connection between matrix factorizations and singularity categories as stated by Polishchuk and Vaintrob in \cite{PV}. Let $X$ be an algebraic stack and $W \in H^0(X,\C)$ a potential. Assume that $W$ is not a zero divisor (i.e. the morphism $W : \Ocal_X \to \C$ is injective). Set $X_0 :=W^{-1}(0)$. Then there is a natural functor (see \cite[Section 3]{PV}).
\begin{gather}
\mathfrak{C} : \operatorname{HMF}(X,W) \to D_{sg}(X_0),\\
(E_\bullet, \delta_\bullet) \mapsto \coker(\delta_1 : E_1 \to E_0).
\end{gather}

\begin{defin}
\begin{enumerate}[(i)]
\item $X$ has the resolution property (RP) if for every coherent sheaf $\Fcal$ on $X$ there exists a vector bundle $V$ on $X$ and a surjection $V \to \Fcal$.
\item $X$ has finite cohomological dimension (FCD) if there exists an integer $N$ such that for every quasi-coherent sheaf $\Fcal$ on $X$ one has $H^i(X,\Fcal)=0$ for $i >N$.
\item $X$ is called a FCDRP-stack if it has property (i) and (ii).
\end{enumerate}
\end{defin}

\begin{thm}\cite[Theorem 3.14]{PV} \label{MFSing}
Let $X$ be a smooth FCDRP-stack and $W$ a potential which is not a zero-divisor. Then the functor
\[ \bar{\mathfrak{C}} : \operatorname{DMF}(X,W) \to D_{sg}(X_0) \]
induced by $\mathfrak{C}$ is an equivalence of triangulated categories.
\end{thm}

\begin{prop} \label{FCDRP}
Let $U$ be a Noetherian scheme and $G$ a reductive algebraic group acting on it. Assume that $U$ has an ample family of $G$-equivariant line bundles. Then the quotient stack $U/G$ is a FCDRP-stack.
\end{prop}
\begin{proof}
See \cite[Section 3]{PV}.
\end{proof}

\section{Application to the setting from Isik}

Recall the notation from section \ref{Aside}: $X$ is a smooth complex algebraic variety with an action of a reductive algebraic group $G$. We also have a $G$-equivariant vector bundle  $\pi : E \to X$ of rank $n$. Its sheaf of $G$-equivariant sections is denoted by $\Ecal$. Let $s \in H^0(X, \Ecal)$ be a regular $G$-equivariant section. The zero scheme of $s$ is denoted by $Y$. We also defined the function.
\[ W: E^\vee \to \C, \qquad a_x \mapsto \langle ((\pi^\vee)^*s)(a_x), a_x \rangle. \]
We assume that $W$ is not a zero divisor. The section $s$ is $G$-equivariant by assumption and the pairing is $G$-invariant so $W$ is $G$-invariant. The $\C^*$-action is given by dilation of the fibers and $W$ is also invariant with respect to this action. Thus, $W$ factors through the quotient
\[ \xymatrix{E^\vee \ar[rr]^W \ar[rd]& & \C \\ & E^\vee/(G \times \C^*) \ar[ru]_{\overline{W}} &} \]
For any linear algebraic group $K$ acting on a scheme $V$ there is an equivalence $\Coh^{K}(V) \iso \Coh(V/K)$. This induces equivalences
\[D^{G \times \C^*}_{sg}(W^{-1}(0)) \iso D_{sg}(W^{-1}(0)/(G \times \C^*)) \iso D_{sg}(\overline{W}^{-1}(0)).\]

\begin{prop}
With the above assumptions there is an equivalence of categories
\[\operatorname{DMF}(E^\vee/(G \times \C^*),\bar{W}) \iso D_{sg}(\overline{W}^{-1}(0)).\]
\end{prop}
\begin{proof}
We need to show that the assumptions of theorem \ref{MFSing} are satisfied. By assumption $X$ is smooth, hence $E^\vee$ is smooth so the stack $E^\vee /(G \times \C^*)$ is smooth. The variety $X$ has an ample family of $G \times \C^*$-equivariant line bundles. The pull-back of such a family along the $G \times \C^*$-equivariant bundle map $\pi^\vee : E^\vee \to X$ is an ample family of $G \times \C^*$-equivariant line bundles on $E^\vee$. It then follows from proposition \ref{FCDRP} that $E^\vee/(G \times \C^*)$ is a FCDRP-stack.
\end{proof}

By remark \ref{MFQuotient} we have
\[ \operatorname{DMF}(E^\vee/(G \times \C^*),\overline{W}) \iso \operatorname{DMF}_{G \times \C^*}(E,^\vee W). \]
Combining these equivalences with corollary \ref{ZCoh} we obtain

\begin{thm}\label{DMFY}
There is an equivalence of categories
\[ \operatorname{DMF}_{G \times \C^*}(E^\vee,W) \iso D^b(\Coh^G(Y)). \]
\end{thm}

\subsection{Extension to an arbitrary linear algebraic group}
In this section we extend theorem \ref{DMFY} to arbitrary linear algebraic groups. Let $X$ be a smooth complex algebraic variety with an action of a linear algebraic group $H$. Let $\pi: E \to X$ be a $H$-equivariant vector bundle and $s$ a $H$-equivariant regular section. Let $Y$ denote the zero section of $s$ and let $W$ be the function
\[ W: E^\vee \to \C, \qquad a_x \mapsto \langle ((\pi^\vee)^*s)(a_x), a_x \rangle. \]

\begin{thm} \label{ThmLinAlg}
With the above assumptions there is an equivalence of categories
\[ \operatorname{DMF}_{\C^*}(E^\vee/H,\bar{W}) \iso D^b(\Coh^H(Y)). \]
\end{thm}
\begin{proof}
Embed $H$ into a reductive algebraic group $G$. Then $H$ acts freely on $G \times X$ by $h \cdot (g,x):=(gh^{-1}, hx)$. Consider the quotient by this action. The morphism
\[ \pi_G : \tfrac{G \times E}{H} \to \tfrac{G \times X}{H}, \qquad (g,e) \mapsto (g, \pi(e)). \]
is a $G$-equivariant vector bundle. Consider the section
\[ s_G :  \tfrac{G \times X}{H} \to \tfrac{G \times E}{H}, \qquad (g,x) \mapsto (g, s(x)).\]
The zero scheme for this section is $\tfrac{G \times Y}{H} =: Y_G$. We have the corresponding function.
\[ W_G : \left( \tfrac{G \times E}{H} \right)^\vee =\tfrac{G \times E^\vee}{H} \stackrel{\id \times s_G}{\longrightarrow} \tfrac{G \times E^\vee}{H} \times_{\tfrac{G \times X}{H}} \tfrac{G \times E}{H} \stackrel{\langle \;, \; \rangle}{\longrightarrow} \C. \]
Notice that $W_G(g,e)=W(e)$ for all $(g,e) \in \tfrac{G \times E}{H}$. Inserting this into theorem \ref{DMFY} gives an equivalence
\[ \operatorname{DMF}_{G \times \C^*}(\tfrac{G \times E^\vee}{H},W_G) \iso D^b(\Coh^G(Y_G)). \]
Notice that $D^b(\Coh^G(Y_G)) \iso D^b(\Coh^H(Y))$ and that
\begin{align} 
\operatorname{DMF}_{G \times \C^*}(\tfrac{G \times E^\vee}{H},W_G) & \iso \operatorname{DMF}_{\C^*}(\tfrac{G \times E^\vee}{H} {\big/} G,\bar{W}_G)\\
& \iso \operatorname{DMF}_{\C^*}(E^\vee / H,\bar{W}).
\end{align}
This finishes the proof.
\end{proof}

\subsection{Application to Hamiltonian reduction}
Let $X$ be a smooth complex algebraic variety with a free action of a linear algebraic group $G$ such that the quotient $Y:=X/G$ is a scheme. The moment map $\mu : T^*X \to \g^*$ provides a $G$-equivariant section of  the trivial vector bundle $\pi : T^* X \times \g^* \to T^*X$. The potential $W$ is the composition
\[ W : T^*X \times \g \stackrel{\mu \times \id}{\longrightarrow} \g^* \times \g \stackrel{\langle \:,\: \rangle}{\longrightarrow} \C.\]
Theorem \ref{ThmLinAlg} gives an equivalence of categories between matrix factorizations and the Hamiltonian reduction
\[ \operatorname{DMF}_{\C^*}((T^*X \times \g)/G, \bar{W}) \iso D^b(\Coh(\mu^{-1}(0)/G)). \]
Write $\mu^{-1}(0)$ as $T^*X \times_{\g^*} 0$. Let $p : X \to X/G$ be the quotient morphism and $\{U_i\}$ a trivialization. Then
\begin{align}
T^*(p^{-1}(U_i)) \times_{\g^*} 0 \iso (G \times \g^* \times T^*U_i) \times_{\g^*} 0 \iso G \times T^*U_i.
\end{align}
This shows that $\mu^{-1}(0)/G \iso T^*(X/G)$. In particular, we proved that

\begin{thm}
There is an equivalence of categories
\[ \operatorname{DMF}_{\C^*}((T^*X \times \g)/G,\bar{W}) \iso D^b(\Coh(T^*(X/G))). \]
\end{thm}

\begin{rem}
When $X$ is a reductive algebraic group $G$ and the linear algebraic group from the theorem is a Borel subgroup $B$ then we get the Springer resolution $\niltil \iso T^*(G/B)$. Thus, we get $ \operatorname{DMF}_{\C^*}((T^*G \times \bfr)/B,\bar{W}) \iso D^b(\Coh^G(\niltil))$. 
\end{rem}

\end{document}